\documentclass[11pt]{amsart}
\usepackage{geometry}                
\usepackage[parfill]{parskip}    
\usepackage{graphicx}
\usepackage{amssymb}
\usepackage{epstopdf}
\usepackage{dsfont}
\usepackage{indentfirst}
\usepackage{etoolbox}
\usepackage{hyperref}
\makeatletter
\newcommand*{\rom}[1]{\expandafter\@slowromancap\romannumeral #1@}
\makeatother
\makeatletter
\patchcmd{\@maketitle}
  {\ifx\@empty\@dedicatory}
  {\ifx\@empty\@date \else {\vskip3ex \centering\footnotesize\@date\par\vskip1ex}\fi
   \ifx\@empty\@dedicatory}
  {}{}
\patchcmd{\@adminfootnotes}
  {\ifx\@empty\@date\else \@footnotetext{\@setdate}\fi}
  {}{}{}
\makeatother

\newtheorem{thm}{Theorem}[section]
\newtheorem{cor}[thm]{Corollary}
\newtheorem{lem}[thm]{Lemma}
\newtheorem{prop}[thm]{Proposition}
\newtheorem{claim}[thm]{Claim}
\theoremstyle{definition}
\newtheorem{defn}[thm]{Definition}
\theoremstyle{remark}
\newtheorem{rmk}[thm]{Remark}

\numberwithin{equation}{section}

\begin{document}
\title{Removable singularities of the cscK metric}
\author{Yu Zeng}
\address{Department of Mathematics, Stony Brook University, 100 Nicolls Rd, Stony Brook, NY, 11794}
\email{yu.zeng@stonybrook.edu}
\maketitle

\begin{abstract}
In this paper, we prove a removable singularity theorem for cscK(constant scalar curvature K\"ahler) metrics, which generalizes the result of Chen-He\cite[Theorem 6.2]{Chen}. Let $f$ be a holomorphic function on $\mathbb{D}^n$, and denote $S = \{ f=0\}$. Suppose $\varphi \in  C_{loc}^{\infty}(\mathbb{D}^n \setminus S) \cap L^{\infty}(\mathbb{D}^n) \cap PSH(\mathbb{D}^n \setminus S)$ defines a cscK metric 
$$g_\varphi=\sqrt{-1} \varphi_{i \bar{j}}\mathrm{d}z^i\wedge\mathrm{d}\bar z^j$$ 
on $\mathbb{D}^n \setminus S$. 
If there exist a constant $C>0$ and a function $\theta(r) = o(1)$ as $r \rightarrow 0$ such that
\begin{equation*}
\frac{1}{C} \exp\{- \theta(|f|)\big|\log |f|\big|^{\frac{1}{2}}\} I \leq g_{\varphi} \leq C I
\end{equation*} 
where $I = \sqrt{-1} \delta_{i\bar j} \mathrm d z^i \wedge \mathrm{d} \bar z^j$, then $g_{\varphi}$ extends to  a smooth cscK metric on $\mathbb{D}^n$.
\end{abstract}

\section{Introduction}
An inspiring example about the removable singularities of cscK (constant scalar curvature K\"ahler) metric is the isolated singularity in complex dimension one. Suppose $\varphi \in C_{loc}^{\infty}(\mathbb{D}^*) \cap L^{\infty}(\mathbb{D}) \cap PSH(\mathbb{D}^*)$ and $\Delta \varphi > 0 $ on $\mathbb{D}^*$ where $\mathbb{D}$ is the unit disk in $\mathbb{C}$ and $\mathbb{D}^* = \mathbb{D} \setminus \{0\}$. Let $u = \log (\Delta \varphi)$. Then the K\"ahler metric $\sqrt{-1} \partial \bar \partial \varphi$ on $\mathbb{D}^*$ has constant scalar curvature equal to $K$ if and only if
\begin{align}
\Delta u = -K e^u.
\end{align}
If $u = o(1) \log r$, we can extend $\varphi$ to be a smooth K\"ahler potential on $\mathbb{D}$. This follows from a simple argument using the weak maximum principle of the Laplacian equation. If $ u \sim \log r $, then $\varphi$ may not extend smoothly, for example, the standard conical K\"ahler metric.

Removable singularity problems in higher dimensions can be very complicated, for instance, see the recent work of LeBrun \cite{CB} and Chen-Donaldson-Sun\cite{CDS}. The present paper is a generalization to the last section of Chen-He\cite{Chen} where they can remove isolated singularity of a cscK metric which is quasi-isometry to a smooth background K\"ahler metric.

In this paper we'll prove the following main theorem. 
\begin{thm}[Main Theorem, see Corollary $\ref{thm5}$]\label{thm1.1}
Let $f$ be a holomorphic function on $\mathbb{D}^n$, and denote $S = \{ f=0\}$. Suppose $\varphi \in  C_{loc}^{\infty}(\mathbb{D}^n \setminus S) \cap L^{\infty}(\mathbb{D}^n) \cap PSH(\mathbb{D}^n \setminus S)$ and it defines a cscK metric $g_\varphi = \sqrt{-1} \varphi_{i\bar{j}}\mathrm{d}z^i\wedge\mathrm{d}\bar{z}^j$ on $(\mathbb{D}^n \setminus S)$. 
If there exist a constant $C>0$ and a function $\theta(r) = o(1)$ as $r \rightarrow 0$ such that 
\begin{align}\label{eqn1.1}
\frac{1}{C} \exp\{-\theta(|f|)\big|\log |f|\big|^{\frac{1}{2}}\} I  \leq g_{\varphi} \leq C I 
\end{align} 
where $I = \sqrt{-1} \delta_{i\bar j} \mathrm d z^i \wedge \mathrm d \bar z^j $, then $g_{\varphi}$ extends to  a smooth cscK metric on $\mathbb{D}^n$.
\end{thm}

\begin{rmk}\label{rmk1.2}
The assumptions, $\varphi \in L^{\infty}(\mathbb{D}^n)$ and $ 0 \leq \Delta \varphi \leq C$ holding in $C^{\infty}$ sense on $\mathbb{D}^n \setminus S$, imply that $\varphi \in W^{2,p}(\mathbb{D}^n)$ for any $1 < p < \infty$. Thus $\varphi \in C^{1,\alpha}(\mathbb{D}^n)$. This fact will be proved in Theorem $\ref{thm7}$
\end{rmk}

\begin{rmk}
Condition $(\ref{eqn1.1})$ rules out the possibilty that $g_{\varphi}$ has conical singularities.
\end{rmk}

To prove the main theorem, we view the cscK equation as two equations
\begin{align}
&\Delta_{\varphi} u := g_{\varphi}^{i\bar j} \frac{\partial^2 u}{ \partial z_i \partial \bar z_j}  = - K, \label{eqn1.2}\\
&u = \log \det(g_{\varphi}), \label{eqn1.3}
\end{align}
where K is the constant scalar curvature. In order to avoid unnecessary difficulty and complication, we'll first work on a simpler case when our singular locus is a smooth divisor in Secion $\ref{sec1}$ and Secion $\ref{sec2}$. In Section $\ref{sec1}$, we'll use techniques in \cite{Trudinger2} to deal with $(\ref{eqn1.2})$ which is a linear elliptic equation with measurable coefficients and improve the regularity of $u$ to $C^{\alpha}$. In Section $\ref{sec2}$, we deal with $(\ref{eqn1.3})$, which is the well-known complex Monge-Amp\`ere equation. It's a fully nonlinear elliptic equation. Recently, Y. Wang\cite{YW} has developed the $C^{2,\alpha}$-estimate for the complex Monge-Amp\`ere equation assuming potential lying in $C^2$. The main idea in his work is to view the solution of complex Monge-Amp\`ere equation as a solution of a real fully nonlinear elliptic equation. He also suggests that the argument also works for the viscosity solution. In Section $\ref{sec2}$, we'll first show that $\varphi \in W^{2,p}_{loc}(\mathbb{D}^n)$ for $p>n$ and then show that a $W^{2,p}$ solution of complex Monge-Amp\`ere is also the viscosity solution for the real nonlinear elliptic equation as described in Wang's paper. Thus, we get $C^{2,\alpha}$ regularity of $\varphi$.  In Section $\ref{sec3}$, we prove the removable singularity theorem of cscK metric for higher codimension singular set with real codimension $d >2$. And also, as a corollary, we give the proof of the Theorem $\ref{thm1.1}$.

\section {The linear problem} \label{sec1}

Let $u = \log \det(g_{\varphi}) \in C_{loc}^{\infty}(\mathbb{D}^n \setminus \{z_1=0\})$. It satisfies the linear elliptic equation 
\begin{align}\label{eqn11}
-\Delta_{\varphi} u = K
\end{align}
in the $C^{\infty}$ sense on $\mathbb{D}^n \setminus \{z_1=0\}$, since $g_{\varphi}$ is a smooth cscK metric on $\mathbb{D}^n \setminus \{z_1=0\}$. However, notice that $(\ref{eqn11})$ is not a strongly and uniformly elliptic equation as in the classical elliptic theory. Therefore, we use methods in \cite{Trudinger2} which treat the ellipitic equations with only measurable coefficients. 

We introduce a few definitions from \cite{Trudinger2}. Let $\mathcal{A} = [a^{i\bar{j}}] $ be a positive definite, measurable, $n\times n$ Hermitian matrix valued function on $\mathbb{D}^n$ and $\mu$ a nonnegative measurable function on $\mathbb{D}^n$. Denote $\Lambda_{\mathcal{A}}$ and $\lambda_{\mathcal{A}}$ the largest and smallest eigenvalues of $\mathcal{A}$. Under the assumptions that $\Lambda_{\mathcal{A}}$ and $\mu$ belong to $L^1(\mathbb{D}^n)$ and $\mu$ is positive on a subset of $\mathbb{D}^n$ of positive measure,  the formula
\begin{align}
\langle \varphi, \psi \rangle_{\mathcal{A}} = \int_{\mathbb{D}^n} (a^{i\bar{j}}\frac{\partial \varphi}{\partial z_i} \frac{\partial \psi}{\partial \bar{z_{j}}} + \mu \varphi \psi)dx  \label{eqn12}
\end{align}
defines a real inner product on $C^{\infty}(\mathbb{D}^n)$. 

\begin{defn}
The Hilbert space $H(\mathcal{A}, \mu, \mathbb{D}^n)$, is defined as the completion of $C^{\infty} (\mathbb{D}^n)$ under the inner product $(\ref{eqn12})$.
\end{defn}

\begin{rmk}
If $\Lambda_{\mathcal{A}}$ and $\mu$ are not integrable, then $\langle , \rangle_\mathcal{A}$ is not well defined. Fortunately, in our case $\mu = \Lambda_{\mathcal{A}} \in L^{\infty}(\mathbb{D}^n)$. 
\end{rmk}

\begin{rmk}
If $\lambda_{\mathcal A}^{-1} \in L^1(\mathbb{D}^n)$, then $\mathcal H(\mathcal A, \mu, \mathbb{D}^n)$ is a subspace of $W^{1,1}(\mathbb{D}^n)$.
\end{rmk}
For simplicity, write $H(\mathcal{A},\mathbb{D}^n) = H(\mathcal{A},\Lambda_{\mathcal{A}}, \mathbb{D}^n)$.

\begin{defn}
$u$ is called a $H(\mathcal{A} , \mathbb{D}^n)$ solution of equation $-\frac{\partial }{\partial z_i} (a^{i\bar{j}} \frac{\partial u}{\partial \bar{z_{j}}}) = f$ with $f\in L^1(\mathbb{D}^n)$, if $u \in H(\mathcal{A}, \mathbb{D}^n)$ and for any $\varphi \in C_{0}^{\infty}(\mathbb{D}^n)$, we have
\begin{align}
\int_{\mathbb{D}^n} a^{i\bar{j}}\frac{\partial u}{\partial z_i} \frac{\partial \varphi}{\partial \bar{z_j}}dx = \int_{\mathbb{D}^n} f\varphi dx.
\end{align}
\end{defn}

For the rest of this section, we denote $\mathcal{A} =[a^{i\bar{j}}]= [g_{\varphi}^{i\bar{j}} \det(g_{\varphi})]$ and denote $r\mathbb{D}^n = \{|z_i | < r\}.$ Using notions above, we can conclude the following.
\begin{prop}\label{thm1}
Let $u$ satisfy $(\ref{eqn11})$ in the $C^{\infty}$ sense on $\mathbb{D}^n \setminus \{z_1 =0\}$. If there exist a constant $C>0$ and a function $\theta(r) = o(1)$ as $r \rightarrow 0$, s.t. 
\begin{align}\label{eqn13}
\frac{1}{C} \exp\{- \theta(|z_1|) \big| \log |z_1| \big|^{\frac{1}{2}}\} I \leq g_{\varphi} \leq C I ,
\end{align} 
then u  is a $H(\mathcal{A}, 0.5\mathbb{D}^n)$ solution of the equation  
\begin{align}\label{eqn15}
-\frac{\partial }{\partial z_i} (a^{i\bar{j}} \frac{\partial u}{\partial \bar{z_{j}}}) = K \det 
g_{\varphi}.
\end{align}
\end{prop}

Before we prove Proposition $\ref{thm1}$, we need to introduce the logarithmic cut-off function which also appears in \cite{CDS}. For $\epsilon >0$, let 
\begin{equation*}
\eta_{\epsilon}(z)=\begin{cases}
1 &    |z|< \epsilon^2, \\
\frac{1}{\log \epsilon}(\log |z| -\log\epsilon) &   \epsilon^2 \leq |z| \leq  \epsilon,\\
0  &     |z| > \epsilon.
\end{cases}
\end{equation*}
Immediately we know that $\eta_{\epsilon}\in W_{0}^{1,2}(\mathbb{D})$ and
\begin{equation*}
\int_{\mathbb{D}} |\nabla \eta_{\epsilon}|^2 dx = \frac{2\pi}{\log^2 \epsilon} \int_{\epsilon^2}^{\epsilon} \frac{1}{r^2} r dr = -\frac{2\pi}{\log \epsilon}.
\end{equation*}

Now we are ready to prove Proposition $\ref{thm1}$.
\begin{proof}
First, we claim that 
\begin{equation}\label{eqn14}
\int_{0.5\mathbb{D}^n} a^{i\bar{j}}\frac{\partial u}{\partial z_i} \frac{\partial u}{\partial \bar{z_j}}dx = \int_{0.5\mathbb{D}^n}|\nabla_{\varphi} u|_{\varphi}^2 \det( g_{\varphi} )dx < +\infty,
\end{equation}
where $0.5\mathbb{D}^n = \{|z_i|<0.5\}$.
Let $\sigma_{\epsilon} = (1- \eta_{\epsilon}) \psi \in W_{0}^{1,2}(\mathbb{D}^n)$, which is supported away from the divisor $\{z_1= 0\} \subset \mathbb{C}^n$, where $\psi$ is a smooth cut-off function with $\psi =1$ in $0.5\mathbb{D}^n$ and supported in $0.75\mathbb{D}^n$. We have
\begin{equation*}
|\nabla (\sigma_{\epsilon} u)|_{\varphi}^2 = \langle \nabla (\sigma_{\epsilon}^2  u ), \nabla u \rangle_{\varphi} + u^2|\nabla \sigma_{\epsilon}|_{\varphi}^2.
\end{equation*}
Therefore,
\begin{align*}
\int_{\mathbb{D}^n}|\nabla ( \sigma_{\epsilon} u) |_{\varphi}^2 \det(g_{\varphi}) dx 
=\underbrace{\int_{\mathbb{D}^n}\langle \nabla (\sigma_{\epsilon}^2 u), \nabla u\rangle_{\varphi} \det(g_{\varphi})dx}_{I} 
+ \underbrace{\int_{\mathbb{D}^n}u^2 |\nabla \sigma_{\epsilon}|_{\varphi}^2\det(g_{\varphi})dx }_{II} .
\end{align*}
By integration by parts, we get
\begin{align*}
I=   \int_{\mathbb{D}^n} K  \sigma_{\epsilon}^2 u  e^u dx \leq   C.
\end{align*}
According to $(\ref{eqn13})$, $u^2 \leq \theta(|z_1|)\log |z_1| +C $. And it implies that $u \in L^2_{loc}(\mathbb{D}^n)$. Thus,
\begin{align*}
II   \leq  \int_{\mathbb{D}^n} u^2  |\nabla \sigma_{\epsilon}|^2 \Lambda_{\mathcal{A}}dx
    &\leq C \{\int_{0.75 \mathbb{D}^n} u^2 |\nabla \eta_{\epsilon}|^2 dx + \int_{0.75\mathbb{D}^n} u^2 |\nabla \psi|^2 dx\} \\
&\leq C\{ \int_{0.75\mathbb{D}^{n-1}} (\frac{1}{\log^2 \epsilon}\int_{0}^{2\pi}\int_{\epsilon^2}^{\epsilon}   \frac{u^2}{r} drd\theta)  dy + \int_{0.75\mathbb{D}^n} u^2 dx\}\\
& \leq  C\{ \frac{1}{\log^2 \epsilon}\int_{2 \log \epsilon}^{\log \epsilon}   \big(\theta(r)\log r +C\big) d(\log r)   + 1\}\\
& \leq C(o(1) + 1) \leq C.
\end{align*}

So as $\epsilon \rightarrow 0$, by Fatou's lemma we have
\begin{align*}
\int_{0.5\mathbb{D}^n}|\nabla_{\varphi} u|_{\varphi}^2 \det( g_{\varphi} )dx &\leq \int_{\mathbb{D}^n}|\nabla_{\varphi} (\psi u)|_{\varphi}^2 \det( g_{\varphi} )dx \\
&\leq \liminf_{\epsilon \rightarrow 0} \int_{\mathbb{D}^n}|\nabla ( \sigma_{\epsilon} u) |_{\varphi}^2 \det(g_{\varphi}) dx  \leq C.
\end{align*}
So the claim is proved. However, one thing to notice is that $(\ref{eqn14})$ doesn't mean that $u \in H(\mathcal{A}, 0.5\mathbb{D}^n)$. This is because that $[a^{i\bar{j}}]$ is not strongly and uniformly elliptic, the completion of smooth functions under the norm $(\ref{eqn12})$ is in general a strictly smaller subset of functions with integrable derivatives for which $(\ref{eqn12})$ is finite. Therefore, we need to verify that $u$ can be approximated under the norm $(\ref{eqn12})$ by smooth functions. Fortunately, $(1- \eta_{\epsilon})u$'s for $\epsilon \rightarrow 0$ are what we need. To justify this, we can go to the the estimation of $II$ and use the same trick to prove that
 $$ \|u - (1- \eta_{\epsilon}) u \|_{\mathcal{H}(\mathcal{A}, 0.5 \mathbb{D}^n)}= \int_{0.5\mathbb{D}^n} |\nabla (\eta_{\epsilon} u )|_{\varphi}^2 \det(g_{\varphi}) dx \rightarrow 0,$$ as $\epsilon \rightarrow 0 $.
\end{proof}

Applying techniques in \cite{Trudinger2}, we have an immediate corollary, 
\begin{cor}\label{thm10} 
Let u be a $H(\mathcal{A}, 0.5\mathbb{D}^n)$ solution of $(\ref{eqn15})$. Then $u \in L_{loc}^{\infty} (0.5 \mathbb{D}^n)$.
\end{cor}
\begin{proof}
We just need to check that $\lambda_{\mathcal{A}}^{-1} \in L_{loc}^{t}(\mathbb{D}^n)$ for some $t > n$. By assumption $(\ref{eqn13})$,
\begin{align*}
\lambda_{\mathcal{A}}^{-1} \leq C [\det(g_{\varphi})]^{-1} \leq C(\frac{1}{|z_1|})^{\theta(|z_1|)}.
\end{align*} 
Thus, $\lambda_{\mathcal{A}}^{-1} \in L^{t}(\mathbb{D}^n)$ for all $t < \infty$. Therefore, by Theorem 5.1 and Corollary 5.4 in \cite{Trudinger2}, we conclude that $u \in L_{loc}^{\infty}(\mathbb{D}^n)$.
\end{proof}

Since $g_{\varphi} \leq C I$ and by Corollary $\ref{thm10}$ $\det(g_{\varphi}) = e^u \geq\epsilon $, we conclude that $C^{-1} I \leq g_{\varphi} \leq C I$. Therefore, we can apply the classical elliptic theory \cite[Theorem 8.24] {Trudinger1} to conclude that $u \in C^{\alpha}(\mathbb{D}^n)$ for some $0< \alpha <1$.

\section{The nonlinear problem}\label{sec2}
Let's summarize what we get so far. In last section, assuming condition $(\ref{eqn13})$, we have derived the following consequences:
\begin{enumerate}
\item $u = \log \det(g_{\varphi}) \in C^{\alpha}(\mathbb{D}^n)$ for some $0< \alpha < 1$.
\item $\varphi \in L^{\infty}(\mathbb{D}^n) \cap PSH(\mathbb{D}^n \setminus \{z_1 = 0\}) \cap C_{loc}^{\infty}(\mathbb{D}^n \setminus \{z_1 = 0\})$. 
\item $C^{-1}\leq \Delta \varphi \leq C$ holds in $C^{\infty}$ sense on $\mathbb{D}^n \setminus \{z_1=0\}$.
\end{enumerate}
In \cite{YW}, Yu Wang derived the $C^{2,\alpha}$ regularity of the complex Monge-Amp\`ere equation with $ u \in C^{\alpha}$ and $\Delta \varphi \in L^{\infty}$.  In his paper, he interpreted the solution of complex Monge-Amp\`ere equation to the solution of a real fully nonlinear elliptic equation with an additional assumption $\varphi \in C^2$ and he also pointed out the same argument should work for the viscosity solution as well. Following his obersevation, we'll first try to get the best regularities of $\varphi$ using (2) and (3), and then show that $\varphi$ is a viscosity solution of the real nonlinear elliptic equation as in his paper. 

\begin{thm}\label{thm7}
Suppose $\varphi \in C_{loc}^{\infty}(\mathbb{D}^n \setminus \{z_1 = 0\}) \cap PSH (\mathbb{D}^n \setminus \{z_1=0\})$ satisfies the followings
\begin{enumerate}
\item $\varphi \in {L^{\infty}(\mathbb{D}^n)} $.
\item $\Delta \varphi =f \geq 0$ holds in $C^{\infty}$ sense on $(\mathbb{D}^n \setminus \{z_1=0\})$ with $f \in L_{loc}^{\infty}(\mathbb{D}^n) \cap C^{\infty}(\mathbb{D}^n \setminus \{z_1=0\})$.
\end{enumerate}
Then $\varphi \in W_{loc}^{2,p}(\mathbb{D}^n) \cap PSH(\mathbb{D}^n)$, for any $1 < p < \infty$.
\end{thm}

\begin{proof}
Let $\Gamma(x) = \frac{1}{4n(1-n)\omega_{2n}} |x|^{2-2n}$ for $n > 1$. We introduce Newton Potential from \cite{Trudinger1}. Given any $1 < p <\infty$ and $g \in L^p(\mathbb{D}^n)$, define
\begin{equation*}
Ng(x) = \int_{\mathbb{D}^n} \Gamma(x-y)g(y) dy.
\end{equation*}
$Ng(x)$ is called the Newton Potential of $g$. From \cite[Theorem 9.9]{Trudinger1}, $Ng \in W_{loc}^{2,p}(\mathbb{D}^n)$ and $\Delta Ng = g$ a.e. on $\mathbb{D}^n$. In our case, $Nf(x) \in W_{loc}^{2,p}(\mathbb{D}^n)$ for any $1<p<\infty$ and $\Delta Nf = f$ holds a.e. on $\mathbb{D}^n$. Since $f\in C_{loc}^{\infty}(\mathbb{D}^n \setminus \{z_1 =0\})$, then $Nf \in C_{loc}^{\infty}(\mathbb{D}^n \setminus \{z_1 =0\})$.

Consider $v = \varphi - Nf \in C_{loc}^{\infty}(\mathbb{D}^n \setminus \{z_1=0\}) \cap L_{loc}^{\infty}(\mathbb{D}^n)$. We claim that $\Delta v = 0$ holds in the $W^{1,2}$ weak sense i.e. $v \in W_{loc}^{1,2}(\mathbb{D}^n)$ and for any $\psi \in C_{0}^{\infty}(\mathbb{D}^n)$, 
\begin{align}\label{eqn16}
\int_{\mathbb{D}^n} \langle \nabla v ,\nabla \psi \rangle dx = 0.
\end{align}
We can prove this by applying methods in proving Proposition $\ref{thm1}$. First, we show that $v\in W_{loc}^{1,2}(\mathbb{D}^n)$. We denote $\sigma_{\epsilon}$ as in Proposition $\ref{thm1}$, then
\begin{align*}
\int_{\mathbb{D}^n} |\nabla (\sigma_{\epsilon} v)|^2 dx \leq C (1 + \int_{\mathbb{D}^n} v^2 |\nabla \sigma_{\epsilon}|^2 dx) \leq C(1+ \|\sigma_{\epsilon}\|_{W^{1,2}(\mathbb{D}^n)}) \leq C.
\end{align*}

Let $\epsilon \rightarrow 0$, we proved $v \in W_{loc}^{1,2}(\mathbb{D}^n)$. Thus, given any $\psi \in C_0^{\infty}(\mathbb{D}^n)$, we consider
\begin{align*}
|\int_{\mathbb{D}^n} \langle \nabla v ,\nabla \psi \rangle dx |&= |\int_{\mathbb{D}^n} \langle \nabla v ,\nabla \big( (1- \eta_{\epsilon})\psi \big)\rangle dx + \int_{\mathbb{D}^n} \langle \nabla v ,\nabla \big( \eta_{\epsilon}\psi\big) \rangle dx|\\
& = C (\int_{supp(\psi)} |\nabla v|^2 dx )^{\frac{1}{2}} (\int \big( \psi^2 |\nabla \eta_{\epsilon}|^2 + \eta_{\epsilon}^2 |\nabla \psi|^2 \big) dx)^{\frac{1}{2}}\\
& \rightarrow 0 \text{ as } \epsilon \rightarrow 0.
\end{align*}

 Therefore, $(\ref{eqn16})$ is true and it implies that $v \in C^{\infty}(\mathbb{D}^n)$ and thus $\varphi \in W_{loc}^{2,p}(\mathbb{D}^n) \hookrightarrow C^{1,\alpha}(\mathbb{D}^n)$. Hence, for any $\vec x \in \mathbb{C}^n \setminus \{0\}$, $x^i \bar x^j \varphi_{i\bar j} \geq 0 $ holds in distributional sense on $\mathbb{D}^n$ since $ x^i \bar x^j \varphi_{i\bar j} \in L_{loc}^{p}(\mathbb{D}^n)$ for $p>1$ and it's smooth and nonnegative away from divisor. From \cite{Gunning} corollary on page 117, we can conlude that $\varphi \in PSH(\mathbb{D}^n)\cap W_{loc}^{2, p}(\mathbb{D}^n)$.
\end{proof}

\begin{thm}\label{generalize}
Assume $\varphi \in PSH(\mathbb{D}^n) \cap  W_{loc}^{2,p}(\mathbb{D}^n)$ where $p > n$ s.t.
\begin{equation*}
\log \det ( \varphi_{i \bar{j}}) =  u
\end{equation*}
a.e. on $\mathbb{D}^n$. If $u \in C^{\alpha}(\mathbb{D}^n)$ and $\triangle \varphi \leq C$ a.e. on $\mathbb{D}^n$, then $\varphi \in C^{2,\beta}(\mathbb{D}^n)$ for any $0 < \beta <\alpha$.
\end{thm}

This is our main regularity result and the smooth extension of cscK metric(Corollary $\ref{thm3}$) will be an easy consequence. In Y.Wang's paper \cite{YW}, the main idea is to view the complex Monge-Amp\`ere equation as a real fully nonlinear elliptic equation. We'll follow his idea and main point in proving the above theorem is to show that $\varphi$ satisfies $F_{\theta}(D^2 \varphi) \in C^{\alpha}$ in the viscosity sense for the elliptic operator $F_{\theta}$ introduced by Wang. Thus, before entering into the proof, let's recall some preliminaries from Wang's paper.

Let $Sym(2n)$ be the space of $2n \times 2n$ real symmetric matrices and $Herm(n)$ be the space of $n \times n$ complex Hermitian matrices.
Fix the following canonical complex structure
\begin{equation*}
J = \left( \begin{array}{cc}
0 & -I\\
I & 0
\end{array}\right),  I  \; \text{is the} \; n\times n \;\text{identity matrix}\;
\end{equation*}
on $\mathbb{R}^{2n}$. Then $Herm(n)$ can be identified with the subspace
\begin{equation*}
\{M : [M, J]= MJ - JM = 0\} \subset Sym(2n),
\end{equation*}
by the map
\begin{equation*}
\iota: H = A + iB \mapsto \left( \begin{array}{cc} A & -B \\B & A\end{array}\right).
\end{equation*}
We view $Herm(n)$ as a subspace of $Sym(2n)$ according to the above identification. The complex structure $J$ gives rise to the canonical projection $p: Sym(2n) \rightarrow Herm(n)$
\begin{equation*}
p: M \mapsto \frac{M + J^t M J}{2}.
\end{equation*}
By a straightforward calculation, we get
\begin{claim}
If $\iota(H) =  p(M)$, $M \in Sym(2n)$ and $H \in Herm(n)$, then $\det_{\mathbb{R}}^{\frac{1}{2}} [p(M)] = \det_{\mathbb{C}}(H)$.
\end{claim}
Denote 
$$F(M) :=  \log \det [p(M)]$$

And $F$ is a concave function on the set $\{M \in Sym(2n): p(M) >0 \}$. The definitions here are slightly different than the ones in Y.Wang's paper.

\begin{defn}
Given $0 <\theta < 1$, let $E_{\theta} \subset Sym(2n)$ consist of matrices $N$ such that
\begin{equation*}
\theta I_{2n} \leq p(N) \leq \theta^{-1} I_{2n}.
\end{equation*}
\end{defn}
Define for all $M \in Sym(2n)$,
\begin{equation*}
 F_{\theta}(M) = \sup\{G(M)| G \text{ is concave on } Sym(2n), G(X) = F(X) \text{ for any } X \in E_{\theta}\}.
\end{equation*}
We introduce another alternative equivalent definition of $F_\theta$ that we'll use later.
\begin{lem}
The following definitions are equivalent
\begin{enumerate}
\item $F_{\theta}(M) = \sup\{G(M)| G \text{ is concave on } Sym(2n), G(X) = F(X) \text{ for any } X \in E_{\theta}\}$.
\item $F_{\theta}(M) = \inf\{tr(p(N) (M - N)) + F(N) | N \in E_\theta\}$.
\end{enumerate}
\end{lem}
\begin{proof}
Denote the function defined using definition (1)(respectively (2)) as $F_1$(respectively, $F_2$). So we need to show that $F_1 = F_2$. First of all, we show that $F_1 \leq F_2$. For any concave function $G$ on $Sym(2n)$ with $G(X) = F(X)$ on $E_{\theta}$, we have that the graph of $G$ is below the tangent plane of $G$ at $N \in E_\theta$, i.e. for any $M \in Sym(2n)$
\begin{align}
G(M) \leq tr(p(N) (M-N)) + F(N) \text{ for any } N\in E_\theta.
\end{align}
Thus $G(M) \leq F_2(M)$ for any $M \in Sym(2n)$ and then $F_1 \leq F_2$.

Next, we show that $F_1 \geq F_2$. We indeed show that $F_2$ defines a concave function on $Sym(2n)$ with $F_2(X) = F(X)$ on $E_\theta$. To show $F_2$ is concave, we consider for $M, L \in Sym(2n)$,
\begin{align*}
&tr(p(N)(\frac{M+L}{2}-N)) + F(N) \\
& = \frac{1}{2}\big(tr(p(N)(M -N )) + F(N)\big) + \frac{1}{2}\big(tr(p(N)(L -N )) + F(N) \big)\\
& \geq \frac{1}{2} \big( F_2(M) + F_2(L)  \big)
\end{align*}
Thus, $F_2(\frac{M+L}{2}) \geq \frac{1}{2} \big( F_2(M) + F_2(L)  \big)$. $F_2 = F$ on $E_\theta$ follows from the fact that $F$ is concave on $E_\theta$. Thus this ends the proof of the lemma.
\end{proof}

Following the idea in \cite[Lemma 3.4]{YW}, we can prove a similar lemma in which we take advantages of the equivalent definition (2) of $F_\theta$ in the above lemma.

\begin{lem}
$F_{\theta}$ is concave and uniformly elliptic in $Sym(2n)$, i.e. there exists $\bar{\theta} > 0$ only depending on $\theta$ such that
\begin{equation*}
\bar{\theta}\| P \| \leq F_{\theta}(M + P) - F_{\theta}(M) \leq \bar{\theta}^{-1}\|P\|, \forall M \in Sym(2n), P \geq 0.
\end{equation*}
Moreover $F_{\theta}(M) =  F(M)$ for all $M \in E_{\theta}$.
\end{lem}

Next, let's define the viscosity solution of equation
\begin{equation}\label{eqn7}
F_{\theta}(D^2 \varphi) = 2u.
\end{equation}

\begin{defn}[\cite{CA}, Definition 2.3]
Let $\Omega$ be a domain in $\mathbb{R}^n$. A continuous function $\varphi$ in $\Omega$ is a viscosity subsolution (resp. viscosity supersolution) of $(\ref{eqn7})$ in $\Omega$, when the following condition holds:
if $x_0 \in \Omega$, $\psi \in C^2(\Omega)$ and $\varphi - \psi$ has a local maximum at $x_0$ then
\begin{equation}
F_{\theta} (D^2 \psi (x_0)) \geq 2u(x_0)
\end{equation}
(resp. if $\varphi - \psi$ has a local minimum at $x_0$ then $F_{\theta}(D^2 \psi (x_0)) \leq 2 u(x_0)$). We say that $\varphi$ is a viscosity solution of $(\ref{eqn7})$ when it is subsolution and supersolution.
\end{defn}

Now we're ready to prove Theorem $\ref{generalize}$.

\begin{proof}

 From the assumptions, we can find a positive number $\theta > 0$ s.t. $\theta  \delta_{i\bar{j}} \leq \varphi_{i\bar{j}} \leq \theta^{-1}\delta_{i\bar{j}}$ a.e. on $\mathbb{D}^n$. From now on, $\theta$ is fixed and we use this fixed number to define $F_{\theta}$. Notice, in our case, $F_{\theta} (D^2{\varphi}) = F(D^2{\varphi}) = 2\log \det (\varphi_{i\bar{j}}) = 2u(x)$ a.e. on $\mathbb{D}^n$.

We first prove that $\varphi$ is a viscosity subsolution of $(\ref{eqn7})$ on $\mathbb{D}^n$. For any $x_0 \in \mathbb{D}^n$ and $\psi \in C^{2}(\mathbb{D}^n)$ s.t.
\begin{enumerate}
\item $\psi (x_0) = \varphi(x_0)$,
\item $\psi(x) \geq \varphi (x) $ in a small neighborhood of $x_0$,
\end{enumerate}
we need to show that $F_{\theta}(D^2 \psi(x_0)) \geq 2u(x_0)$. 

First, we claim that matrix
\begin{align}\label{eqn18}
\psi_{i\bar{j}}(x_0) \geq 0.  
\end{align}
We prove $(\ref{eqn18})$ by contradiction. Suppose $(\ref{eqn18})$ doesn't hold, then without loss of generality, we can assume $\psi_{1\bar{1}}(x_0) < 0$. Therefore, $\psi_{1\bar{1}} < 0$ in a small neighborhood of $x_0$, say $U_{x_0}$. Consider $v(z) = (\varphi - \psi) (x_0 + (z, 0, \cdots , 0)) $, which is a nonpositive subharmonic function on its domain with $v(0) = 0$. Moreover, since $\psi_{1\bar{1}} < 0$, $v$ is a strictly subharmonic function i.e. for any $r > 0$, we have
\begin{align}
0 = v(0) < \frac{1}{2\pi} \int_{0}^{2\pi} v(r e^{i\theta}) d\theta \leq 0
\end{align}
Thus, we get a contradiction.

Denote $\psi_{\epsilon} = \psi + \epsilon |z- x_0|^2$ for $\epsilon >0$. Thus, $(\psi_{\epsilon})_{i\bar{j}}(x_0) \geq \epsilon \delta_{i \bar{j}}$. Here we introduce $\psi_\epsilon $ is just to make $F(D^2 \psi_\epsilon)$ well-defined.

Second, we claim that $F(D^2 \psi_{\epsilon}) (x_0) \geq 2u(x_0)$. Again we prove by contradiction. Suppose not, we can get that $F(D^2 \psi_{\epsilon}) (x) - 2u(x) <0$ in a small neighborhood of $x_0$, saying $U_{x_0}$. Consider for a.e. $x\in U_{x_0}$,
\begin{equation}\label{eqn19}
\begin{split}
0 < 2u(x) - F(D^{2} \psi_{\epsilon}) (x) &=  2 \log \det (\varphi_{i\bar{j}}) (x) - 2 \log \det ((\psi_{\epsilon})_{i \bar{j}}) (x) \\
&= 2\int_{0}^1 \frac{d}{dt}  \log \det \big( \big( t \varphi + (1-t) \psi_{\epsilon}   \big)_{i\bar{j}} \big) (x)dt \\
&= a^{i \bar{j}}(x) \frac{\partial^2}{\partial z_i \partial \bar{z_j}}(\varphi - \psi_{\epsilon})(x), 
\end{split}
\end{equation}
where $[a^{i\bar{j}}(x)]= 2\int_{0}^{1} [ \big(t \varphi + (1-t)\psi_{\epsilon} \big)_{k\bar{l}}(x) ]^{-1} dt$ is a measurable positive definite matrix valued function on $\mathbb{D}^n$ with upper and lower bounds for the matrix $[a^{i\bar{j}}]$. Also notice that $(\ref{eqn19})$ holds in $W^{2,n}$ sense, since $\varphi \in W_{loc}^{2,p}(\mathbb{D}^n)$ for any $p>1$. According to the strong maximum principle (Theorem 9.6 in \cite{Trudinger1}), $\varphi -\psi_{\epsilon} = 0$, contradicting $(\ref{eqn19})$. 

Since $F$ is concave on $P = \{M \in Sym(2n): p(M) > 0 \}$ and $F_{\theta}$ is the supremum of all possible concave extensions of $F$ by definition, $F(M) \leq F_{\theta}(M)$ for any $M \in P$. Therefore,
\begin{equation*}
F_{\theta} (D^2 \psi_{\epsilon}) (x_0)\geq F(D^2 \psi_{\epsilon})(x_0) \geq 2u(x_0).
\end{equation*}
We get that 
\begin{align}
F_{\theta} (D^2 \psi (x_0)+ 2\epsilon I) \geq 2u(x_0).
\end{align}
By the uniform ellipticity of $F_{\theta}$, we can let $\epsilon \rightarrow 0$ to conclude that $F_{\theta}(D^2 \psi)(x_0) \geq 2u(x_0)$.

Next, we prove that $u$ is a viscosity supersolution of $(\ref{eqn7})$ on $\mathbb{D}^n$. Let $\psi \in C^{2}(\mathbb{D}^n)$ satisfy that
\begin{enumerate}
\item $\psi(x_0) = \varphi(x_0)$,
\item $\psi (x) \leq \varphi(x) $ in a small neighborhood of $x_0$,
\end{enumerate}
we need to show that $F_{\theta}(D^2 \psi) (x_0) \leq 2u(x_0)$. By definition of $F_{\theta}$, we have for a.e. $x \in \mathbb{D}^n$,
\begin{equation}\label{eqn20}
\begin{split}
F_{\theta}(D^2 \psi ) (x_0) &\leq tr[p(D^2 \varphi (x))^{-1} p(D^2\psi (x_0) - D^2 \varphi(x))] + F(D^2 \varphi)(x), \\
& = 2 \varphi^{i\bar{j}}(x) (\psi_{i\bar{j}}(x_0) - \varphi_{i\bar{j}}(x)) + 2u(x),\\
& = 2 \varphi^{i\bar{j}}(x) \big(\psi_{i\bar{j}}(x) - \varphi_{i\bar{j}}(x)\big) + 2u(x) + 2\varphi^{i\bar{j}}(x)(\psi_{i\bar{j}} (x_0)- \psi_{i\bar{j}}(x) ).
\end{split}
\end{equation}
Denote $L: W_{loc}^{2,n}(\mathbb{D}^n) \rightarrow L_{loc}^n (\mathbb{D}^n)$, $v \mapsto \varphi^{i\bar{j}}(x) v_{i\bar{j}}(x)$. Suppose $L(\psi - \varphi) >0$ a.e. on $B_{\epsilon}(x_0)$ for some $\epsilon >0$, then by the strong maximum principle, we get a contradiction since $\psi - \varphi$ achieves its maximum in the interior at $x_0$. Therefore, for any $\epsilon > 0$, there exists $O_{\epsilon} \subset B_{\epsilon}(x_0)$ with positive measure s.t. $L(\psi - \varphi) \leq 0$ a.e. on $O_{\epsilon}$.

Consider $(\ref{eqn20})$ on $O_{\epsilon}$ for $\epsilon$ small.
\begin{align}
F_{\theta}(D^2 \psi)(x_0) \leq 2u(x_0) + 2\sup_{x \in O_\epsilon}|u(x) - u(x_0)| + C \sup_{x\in O_\epsilon} |D^2 \psi(x) - D^2\psi (x_0)|.
\end{align}
Since $O_{\epsilon} \subset B_{\epsilon}(x_0)$, we let $\epsilon \rightarrow 0$ and get $F_{\theta}(D^2 \psi)(x_0) \leq 2u(x_0)$.

Since $\varphi$ is a viscosity solution to $(\ref{eqn7})$, by the standard nonlinear elliptic theory (Theorem 6.6 and Theorem 8.1 in \cite{CA}), $\varphi \in C^{2,\beta}(\mathbb{D}^n)$, for any $0 < \beta < \alpha$.
\end{proof}

\begin{cor}\label{thm3}
Suppose $\varphi \in C_{loc}^{\infty}(\mathbb{D}^n \setminus \{ z_1 = 0 \})\cap L^{\infty}_{loc} (\mathbb{D}^n ) \cap PSH (\mathbb{D}^n \setminus S)$ defines a cscK metric $g_\varphi = \sqrt{-1}\sum \varphi_{i\bar{j}}\mathrm{d}z^i\wedge\mathrm{d}\bar{z}^j$ on $(\mathbb{D}^n \setminus \{ z_1 = 0\})$. If there exists $C>0$ and $\theta(r) = o(1)$ as $r\rightarrow 0$ s.t. $\frac{1}{C} \exp\{- \theta(|z_1|)\big| \log |z_1| \big|^{\frac{1}{2}}\} I \leq g_{\varphi} \leq C I$, then $g_\varphi$ extends smoothly to a cscK metric on $\mathbb{D}^n$.
\end{cor}

\begin{proof}
Using Proposition $\ref{thm1}$, Theorem $\ref{thm7}$ and Theorem $\ref{generalize}$, we conclude that $\varphi \in C^{2,\beta}(\mathbb{D}^n)$. And consider $(\ref{eqn11})$, we get $ u \in C^{2,\beta}(\mathbb{D}^n)$. Using Lemma 17.16 in \cite{Trudinger1}, we can get $\varphi \in C^{4,\beta}(\mathbb{D}^n)$. By bootstrapping, we can get that $ \varphi \in C^{\infty}(\mathbb{D}^n)$.
\end{proof}

\section{Singularites with real codimension $d > 2$}\label{sec3}
In this section, we consider the case when the singular locus is a closed set $S$ with real Minkowski codimension $d > 2$. Definitions and properties of Minkowski dimension can be found in Chapter 3 in \cite{KF}.

\begin{thm}\label{thm4}
Suppose $\varphi \in C_{loc}^{\infty}(\mathbb{D}^n \setminus S) \cap L^{\infty}(\mathbb{D}^n) \cap PSH(\mathbb{D}^n \setminus S)$ defines a cscK metric on $(\mathbb{D}^n \setminus S)$. If 
\begin{align}\label{eqn21}
\lambda(z) \delta_{i\bar{j}} \leq \varphi_{i\bar{j}} \leq C\delta_{i\bar{j}} 
\end{align}
with $\lambda^{-1} \in L^{p}(\mathbb{D}^n)$ for some $p > n(n-1)$, then $g_\varphi$ extends smoothly to a cscK metric on $\mathbb{D}^n$.
\end{thm}

\begin{proof}
The proof is essentially the same as what we did for the smooth divisor $\{z_1 = 0\}$ as in the last two sections. First we choose appropriate cut off functions $\eta_{\epsilon}$ for $S$. Denote $S_{\epsilon}$ as the $\epsilon$-neighborhood of $S$ and denote the characteristic function of $S_{\epsilon}$ as $\chi_{\epsilon}$. $\eta_{\epsilon}$ is a  mollification of $\chi_\epsilon$ with $|\nabla \eta_\epsilon| \leq \frac{C}{\epsilon}$.

Applying the same argument as in the proof of Proposition $\ref{thm1}$, we can get for $\sigma_{\epsilon} = \psi (1 - \eta_{\epsilon})$ and $u = \log \det (g_\varphi)$ that
\begin{align*}
\int_{0.5\mathbb{D}^n - S_{2\epsilon}} |\nabla u|_{\varphi}^2 \det(g_{\varphi}) dx &\leq \int_{\mathbb{D}^n}|\nabla (\sigma_{\epsilon} u ) |_{\varphi}^2 \det(g_{\varphi}) dx
\leq C + \int_{S_{2\epsilon} \cap \mathbb{D}^n} u^2   |\nabla \sigma_{\epsilon}|^2 \Lambda_{\mathcal{A}} dx\\
&\leq C + \frac{C}{\epsilon^2}\int_{S_{2\epsilon} \cap \mathbb{D}^n} u^2 dx \\
&\leq C + C (\int_{S_{2\epsilon} \cap \mathbb{D}^n} u^{\frac{2d}{d-2}} dx)^{\frac{d-2}{d}}.
\end{align*}
$\|u\|_{L^q(\mathbb{D}^n)} < \infty$ for any $q>1$, since
\begin{align}
\int_{\mathbb{D}^n} |u|^q dx \leq \int_{\mathbb{D}^n} (n|\log \lambda| +C\big)^q dx \leq C\int_{\mathbb{D}^n} \lambda^{-1} dx  < \infty.
\end{align}

Therefore, letting $\epsilon \rightarrow 0$, we get $\int_{0.5\mathbb{D}^n} |\nabla u |_{\varphi}^2 \det(g_{\varphi}) dx < C$. Following the same argument as we did in Proposition $\ref{thm1}$ we can conclude that $u$ is a $H(\mathcal{A}, 0.5\mathbb{D}^n)$ solution for $(\ref{eqn15})$. 

Next, we want to apply Theorem 5.1 and Corollary 5.4 in \cite{Trudinger2} to conclude that $u \in L_{loc}^{\infty}(\mathbb{D}^n)$. One crucial assumption is that $\lambda_{\mathcal{A}}^{-1} \in L^t_{loc}(\mathbb{D}^n)$ for $t> n$. This is because $\lambda_{\mathcal{A}}^{-1} \leq C\lambda^{-(n-1)}$ which implies that $\lambda_{\mathcal{A}}^{-1} \in L_{loc}^{t}(\mathbb{D}^n)$ for $t> n$. Thus, $(\ref{eqn15})$ is an uniformly elliptic equation as in the classical elliptic theory and we can get that $u \in C^{\alpha}(\mathbb{D}^n)$.

Next, It suffices to show that $\varphi \in W_{loc}^{2,p}(\mathbb{D}^n)\cap PSH(\mathbb{D}^n)$ for some $p >n$ and then we can apply Theorem $\ref{generalize}$. To prove this, we will use similar arguments as in proving Theorem $\ref{thm7}$. It's crucial to show that $\Delta v = 0$ holding in $C^{\infty}$ sense on $(\mathbb{D}^n - S)$ with $\|v\|_{L^{\infty}(\mathbb{D}^n)} < C$ implies that $\Delta v = 0$ in $W^{1,2}$ sense on $\mathbb{D}^n$. It suffices to prove that $ v \in W_{loc}^{1,2}(\mathbb{D}^n)$. Given $\sigma_{\epsilon} $ as above, as $\epsilon \rightarrow 0$ , we have 
\begin{align}
 \int_{\mathbb{D}^n}|\nabla \sigma_{\epsilon} v  |^2  dx
\leq C + \int_{S_{2\epsilon} \cap \mathbb{D}^n} v^2   |\nabla \sigma_{\epsilon}|^2  dx
\leq C + \int_{S_{2\epsilon} \cap \mathbb{D}^n} |\nabla \sigma_{\epsilon}|^2  dx  \leq C.
\end{align}
Therefore, applying Theorem $\ref{generalize}$ and bootstrapping arguments, we can conclude that $g_\varphi$ extends to  a smooth cscK metric on $\mathbb{D}^n$.
\end{proof}

Combining results in Theorem $\ref{thm4}$ and Corollary $\ref{thm3}$, we get the following Corollary. And it completes the proof ot the Main Theorem $\ref{thm1.1}$.
\begin{cor}\label{thm5}
Let $f$ be a holomorphic function on $\mathbb{D}^n$, and denote $S = \{ f=0\}$. Suppose $\varphi \in  C_{loc}^{\infty}(\mathbb{D}^n \setminus S) \cap L^{\infty}(\mathbb{D}^n) \cap PSH(\mathbb{D}^n \setminus S)$ defines a cscK metric $g_\varphi = \sqrt{-1}\sum \varphi_{i\bar{j}}\mathrm{d}z^i\wedge\mathrm{d}\bar{z}^j$ on $(\mathbb{D}^n \setminus S)$. 
If there exist a constant $C>0$ and a function $\theta(r) = o(1)$ as $r \rightarrow 0$, s.t. 
\begin{align}\label{eqn22}
\frac{1}{C} \exp\{- \theta(|f|)\big|\log |f|\big|^{\frac{1}{2}}\} I \leq g_{\varphi} \leq C I,
\end{align} 
then $g_\varphi$ extends to  a smooth cscK metric on $\mathbb{D}^n$.
\end{cor}

\begin{proof}
It suffices to consider the singular locus $\{\nabla f = 0\}$ of the divisor $S$. By Weierstrass Preparation Theorem, we can assume that $f = h(z,w) \prod_{i=1}^N (z - \sigma_i (w)) $ for $(z , w) \in B_{\epsilon}(0) \subset \mathbb{C} \times \mathbb{C}^{n-1}$, where $\sigma_i's$ are holomorphic functions with $\sigma_i(0) =0$ and $h(0,0) \neq 0$. 
\begin{align*}
\lambda^{-1} = C\exp\{ \theta(|f|)|\log (|f|)|^{\frac{1}{2}}\} \leq C  \prod_{i=1}^N (\frac{1}{|z - \sigma_i(w)|})^{\theta(|f|)}.
\end{align*}
Thus, for any $t < \infty$, 
\begin{align*}
\| \lambda^{-1} \|_{L^t (B_{\epsilon}(0) )}^t  &\leq (\sqrt{-1})^n\int_{|w| < \epsilon} [ \int_{|z| < \epsilon } \prod_{i=1}^N (\frac{1}{|z - \sigma_i (w)|})^{t\theta(|f|)} dz\wedge d\bar{z} ]dw\wedge d\bar{w} \\
&\leq C \text{ , if } |f| \text{ is sufficiently small.} 
\end{align*}
Thus we can apply Theorem $\ref{thm4}$ and Corollary $\ref{thm10}$ to conclude that $g_\varphi$ extends smoothly to a cscK metric on $\mathbb{D}^n$.

\end{proof}

\section*{Acknowledgement}
The author is very grateful to his advisor X.X.Chen for his generous help and constant encouragement. Also, the author would like to thank C.LeBrun for the kind comments on an earlier version of this paper. And the author wants to thank Chengjian Yao for many helpful discussions. Thanks also goes to the author's parents, who brought him up and gave him chance to pursue his dream. Finally, the author would like to thank the referee for very helpful and detailed comments on the first version of this paper.

\end{document}